\documentclass{article}
\usepackage[utf8]{inputenc}
\usepackage{fullpage}
\usepackage{amsmath}
\makeatletter
\renewcommand*\env@matrix[1][\arraystretch]{%
	\edef\arraystretch{#1}%
	\hskip -\arraycolsep
	\let\@ifnextchar\new@ifnextchar
	\array{*\c@MaxMatrixCols c}}
\makeatother

\usepackage{amssymb,amsfonts, amsthm}

\usepackage[all,arc]{xy}
\usepackage{mathrsfs}
\usepackage{hyperref}
\usepackage{cleveref}
\usepackage{enumerate}
\usepackage{mathtools}
\usepackage{tikz-cd}
\usepackage{tikz}
\usepackage{sseq}
\usepackage{tabu}
\usepackage[style=alphabetic]{biblatex}
\def\zz{\mathbb{Z}}
\def\qq{\mathbb{Q}}

\def\cc{\mathbb{C}}

\def\pp{\mathbb{P}}

\newcommand{\homoldeg}[3]{H_{#1}\left(#2; #3\right)}

\newcommand{\cohomoldeg}[3]{H^{#1}\left(#2; #3\right)}

\newcommand{\Par}[1]{\left(#1\right)}
\newcommand{\BPar}[1]{\left[#1\right]}
\newcommand{\WBPar}[1]{\left\{#1\right\}}
\newcommand{\Abs}[1]{\left|#1\right|}
\newcommand{\Bracket}[1]{\left\langle #1\right\rangle}

\theoremstyle{plain}
\newtheorem{thm}{Theorem}[section]

\theoremstyle{definition}
\newtheorem{cor}[thm]{Corollary}
\newtheorem{prop}[thm]{Proposition}
\newtheorem{lem}[thm]{Lemma}

\title{Monodromy of the families of del Pezzo and $K3$ surfaces branching over smooth quartic curves}
\author{Ad\'an Medrano Mart\'in del Campo}
\date{\today}

\begin{document}
	\maketitle
	\begin{abstract}
		Two families of surfaces arise from considering cyclic branched covers of $\mathbb{P}^{2}$ over smooth quartic curves. These consist of degree 2 del Pezzo surfaces with a $\mathbb{Z}/2\mathbb{Z}$ action and $K3$ surfaces with a $\mathbb{Z}/4\mathbb{Z}$ action. We compute the monodromy groups of both families. In the first case, we obtain the Weyl group $W\left(E_{7}\right)$, corresponding to the automorphisms of the $56$ lines contained in a degree $2$ del Pezzo surface. In the second case we obtain an arithmetic lattice: the unitary group $U\left(h_{L_{-}}\right)$ of a type $\left(1, 6\right)$ quadratic form over $\mathbb{Z}\left[i\right]$ by building on results of Kondo and Allcock, Carlson, Toledo.
	\end{abstract}
	\section{Introduction}\label{Intro}
	
	Given a smooth degree $d$ curve $C$ in $\pp^{2}$ and an integer $k\mid d$, we may associate to it a degree $k$ cyclic branch cover $S$ of $\pp^{2}$ branching over it. A natural, often difficult, question that arises from this construction is to determine the monodromy representation of the family of surfaces it induces. This has been extensively studied in the case of universal families of a given type of varieties. By considering subfamilies of the universal ones and computing their monodromy representation, one obtains a better global picture of the one associated to the universal family, as well as of its acting fundamental group (see, e.g. \cite{mcmullen}). In this paper we compute the monodromy of branched covers of $\pp^{2}$ whose branch locus is a smooth quartic curve.
	
	The parameter space of homogeneous degree $d$ polynomials  in variables $x, y, z$ is given by
	\[\pp\Par{\text{Sym}^{d}\Par{\cc^{3}}}=\pp^{N\Par{d}}\quad\text{where }N\Par{d}=\frac{d\Par{d+3}}{2}.\]
	The \emph{vanishing locus} of $f\in \pp^{N\Par{d}}$ is defined as the set $V\Par{f}=\WBPar{P\in \pp^{2}\mid f\Par{P}=0}$. The \emph{discriminant locus} is the subset $\Delta_{d}\subset \pp^{N\Par{d}}$ consisting of polynomials whose vanishing locus is singular. The parameter space of \emph{smooth} degree $d$ plane curves in $\pp^{2}$ is therefore defined as
	\[\mathcal{U}_{d}=\pp^{N\Par{d}}\setminus \Delta_{d}.\]
	Let $f\in \mathcal{U}_{d}$ and let $C=V\Par{f}$ be its vanishing locus. Then $\BPar{C}=d \BPar{H}\in \homoldeg{2}{\pp^{2}}{\zz}$, where $\BPar{H}$ is the hyperplane class in $\pp^{2}$. The curve $C$ is a complex codimension $1$ submanifold of $\pp^{2}$, so there exists a cyclic $k$-fold branched cover $X$ of $\pp^{2}$ with branched locus equal to $C$ if and only if $\BPar{C}$ is a multiple of $k$ of $\BPar{H}$ (see \cite{morita}, Proposition 4.10) and this is equivalent to $k\mid d$. It is a classical result (see \cite{zariski}) that the fundamental group of the complement of a smooth degree $d$ curve in $\pp^{2}$ is cyclic of order $d$. In light of this, for $f\in \mathcal{U}_{4}$ consider the cyclic $2$-fold and $4$-fold covers branched over $V\Par{f}$:
	
	\begin{center}
		\begin{tikzcd}
			& X_{f} \arrow[dl] \arrow[dd, "\zz/4\zz"] & \\
			\mathcal{P}_{f} \arrow[dr, "\zz/2\zz"']  & & \\
			& \pp^{2}\arrow[r, hookleftarrow] & V\Par{f} 
		\end{tikzcd}
	\end{center}
	The surface $\mathcal{P}_{f}$ is a degree $2$ del Pezzo surface, and $X_{f}$ is a $K3$ surface. We define the \emph{universal $(2, 4)$-branched and $(4, 4)$-branched covers of $\pp^{2}$} as the fiber bundles
	\[\mathcal{E}_{2, 4}=\WBPar{\Par{P, f}\in\mathcal{P}_{f}\times \mathcal{U}_{4}\mid P\in  \mathcal{P}_{f}}\quad\quad\quad \mathcal{E}_{4, 4}=\WBPar{\Par{P, f}\in X_{f}\times \mathcal{U}_{4}\mid P\in  X_{f}}\]
	given respectively by the fibrations
	
	\begin{center}
		\begin{tikzcd}
			\mathcal{P}_{f}\arrow[r, hookrightarrow] & \mathcal{E}_{2, 4}\arrow[d] & \Par{P, f} \arrow[d, mapsto] & & X_{f}\arrow[r, hookrightarrow] & \mathcal{E}_{4, 4}\arrow[d] & \Par{P, f} \arrow[d, mapsto]\\
			& \mathcal{U}_{4} & f & & & \mathcal{U}_{4} & f
		\end{tikzcd}
	\end{center}
	where the fibers $\mathcal{P}_{f}$ and $X_{f}$ are diffeomorphic to a degree $2$ del Pezzo surface and a $K3$ surface, respectively. Fixing a base point $f\in \mathcal{U}_{4}$, we simply refer to these fibers by $\mathcal{P}$ and $X$. The action of $\pi_{1}\Par{\mathcal{U}_{4}}$ on $\cohomoldeg{2}{\mathcal{P}}{\zz}$ and $\cohomoldeg{2}{X}{\zz}$ induces monodromy homomorphisms
	\[\rho_{2}:\pi_{1}\Par{\mathcal{U}_{4}}\to \text{Aut}\Par{\cohomoldeg{2}{\mathcal{P}}{\zz}}\quad\quad\quad \rho_{4}:\pi_{1}\Par{\mathcal{U}_{4}}\to \text{Aut}\Par{\cohomoldeg{2}{X}{\zz}}.\]
	The images of $\rho_{2}$ and $\rho_{4}$ can be further restricted by noting that the respective intersection forms in $\cohomoldeg{2}{\mathcal{P}}{\zz}$ and $\cohomoldeg{2}{X}{\zz}$ remain invariant under the actions  of $\rho_{2}$ and $\rho_{4}$.
	
	\begin{figure}[!htb]
		\centering
		\includegraphics[scale=0.20]{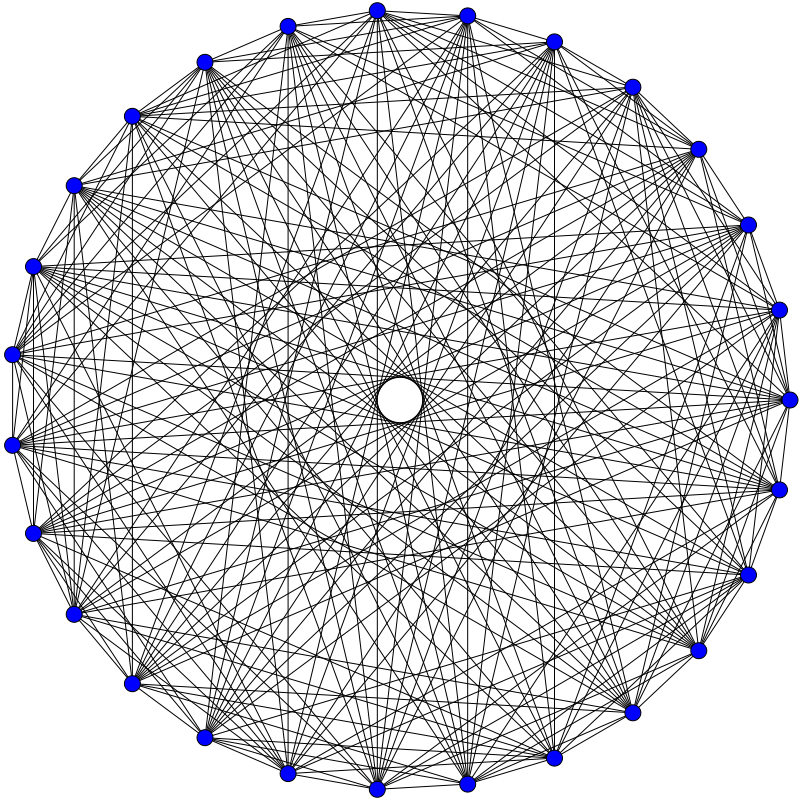}\quad\quad\quad\quad
		\includegraphics[scale=0.20]{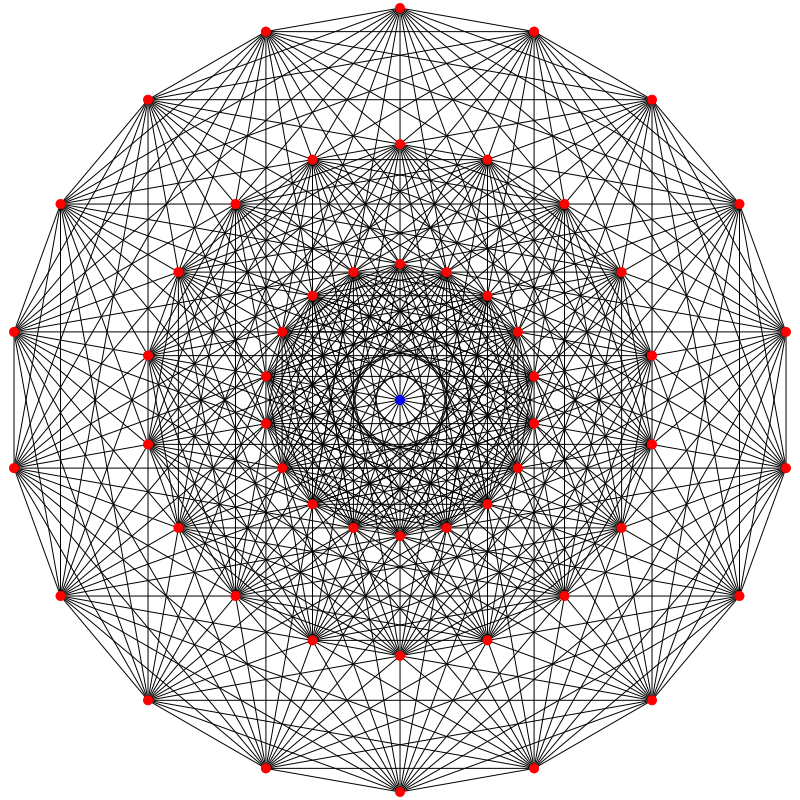}
		\caption{The Schl\"afli graph (left) and Gosset graph (right).}
		\label{SchlafliGosset}
	\end{figure}
	
	As explained in \Cref{delPezzo}, this implies that the image of $\rho_{2}$ is contained in the automorphism group of the $56$ lines contained in $\mathcal{P}$. The intersection pattern of these lines is given by the Gosset graph (see \Cref{SchlafliGosset}), whose automorphism group is $W\Par{E_{7}}$, the Weyl group of $E_{7}$. This fact restricts the image of $\rho_{2}$, so that:
	\[\rho_{2}:\pi_{1}\Par{\mathcal{U}_{4}}\to W\Par{E_{7}}.\]
	This cannot be further restricted, as this map is onto. Namely, we have the following:
	
	\begin{thm}\label{TheoremDelPezzo}
		The monodromy representation
		\[\rho_{2}:\pi_{1}\Par{\mathcal{U}_{4}}\to W\Par{E_{7}}\]
		of the universal $(2, 4)$-branched cover of $\pp^{2}$ is surjective.
	\end{thm}
	Each fiber of $\mathcal{E}_{2,4}$ comes equipped with a $\zz/2\zz$ deck group action induced by its cyclic branched cover structure. The generator $\tau$ of this action is refered to as the \emph{Geiser involution} on $\mathcal{P}$, it corresponds to the generator of the center of $W(E_{7})$ \cite{dolgachev}. As explained in \Cref{delPezzo}, the $56$ lines contained in $\mathcal{P}$ lie in pairs over each of the $28$ bitangents to the underlying quartic curve $V\Par{f}$. The quotient $W\Par{E_{7}}/\Bracket{\tau}\cong \text{Sp}_{6}\Par{\zz/2\zz}$ corresponds to the Galois group of these $28$ bitangents, as discussed in \cite{harris}.
	
	\begin{cor}\label{Geiser}
		The action of $\tau$ on $\cohomoldeg{2}{\mathcal{P}}{\zz}$  is realized as an element of the monodromy group $\text{Im}\Par{\rho_{2}}$.
	\end{cor}
	In order to determine the image of $\rho_{4}$, we make use of the action $T$ on $\cohomoldeg{2}{X}{\zz}$ induced by its $\zz/4\zz$ deck group action. This action determines two sublattices of $\cohomoldeg{2}{X}{\zz}$, namely:
	\[L_{+}=\left\{x\in \cohomoldeg{2}{X}{\zz}\mid T^{2}x=x\right\}\quad\quad\quad L_{-}=\left\{x\in \cohomoldeg{2}{X}{\zz}\mid T^{2}x=-x\right\}.\]
	As discussed in \Cref{Lattices}, these sublattices have discriminant $2^{8}$ and discriminant group (or \emph{glue group} as referred to in \cite{mcmullen2}) isomorphic to $\Par{\zz/2\zz}^{\oplus 8}$ and the action of $T$ on this group is induced by the Geiser involution on $\mathcal{P}$. By considering the branched cover $X\to \mathcal{P}$, we obtain the pullback map
	\[\cohomoldeg{2}{\mathcal{P}}{\zz}\to \cohomoldeg{2}{X}{\zz}\]
	and $L_{+}$ is then generated by the pullbacks of generators of $\cohomoldeg{2}{\mathcal{P}}{\zz}$. As discussed in Section 4 of \cite{artebaniSarti}, the action $T$ is purely non-symplectic, meaning that $T\omega = \pm i\omega$ on the invariant line $\cohomoldeg{2, 0}{X}{\cc}$. The N\'eron-Severi, or Picard, group $\text{NS}\Par{X}:=\cohomoldeg{2}{X}{\zz}\cap \cohomoldeg{1, 1}{X}{\cc}$ of a generic fiber $X$ is then
	\[\text{Pic}\Par{X}\cong \text{NS}\Par{X}\cong L_{+}\cong\Bracket{2}\oplus A_{1}^{\oplus 7}.\]
	This determines the Hodge decomposition on $L_{\pm}\otimes \cc$ since $L_{+}\otimes \cc\subset \cohomoldeg{1, 1}{X}{\cc}$ and therefore $L_{-}\otimes \cc$ contains the lines $\cohomoldeg{2, 0}{X}{\cc}, \cohomoldeg{0, 2}{X}{\cc}$, hence $L_{-}\otimes \cc\cong \cc\oplus \cc^{12}\oplus \cc$.
	
	The monodromy map $\rho_{2}$ then determines the action of $\rho_{4}$ on $L_{+}$, and it remains to study the action of $\rho_{4}$ on $L_{-}$. To do this, we observe that the action of $\rho_{4}$ commutes with the deck group action $T$, therefore having its image lie in the centralizer $C_{L_{-}}\Par{T}$ of $T$ in $L_{-}$. Since $L_{-}$ is equipped with an action $T$ satisfying $T^{2}+I=0$, this centralizer can be regarded as a $\zz\BPar{i}$-module. Following \cite{ACT} and \cite{kondoBook}, we obtain that as a $\zz\BPar{i}$-lattice,
	\[L_{-}\cong\Par{\zz\BPar{i}^{7}, h_{L_{-}}}\]
	where $h_{L_{-}}$ is quadratic form of signature $\Par{1, 6}$ given by (see \Cref{ZiModuleOnL-})
	\[h_{L_{-}}=-2\Par{\Abs{z_{0}}^{2}
		+\Abs{z_{1}}^{2}
		+\Abs{z_{2}}^{2}
		+\Abs{z_{3}}^{2}
		+\Abs{z_{4}}^{2}
		-\Re\Par{{z_{1}\overline{z_{2}}}+{z_{3}\overline{z_{4}}}+{z_{5}\overline{z_{6}}}}
		-\Im\Par{{z_{1}\overline{z_{2}}}+{z_{3}\overline{z_{4}}}+{z_{5}\overline{z_{6}}}}}\]
	Since $\rho_{4}$ preserves the intersection form on $L_{-}$, and therefore the hermitian form $h_{L_{-}}$, it follows that
	\[\text{Im}\Par{\rho_{4}\mid_{L_{-}}}\subseteq U\Par{h_{L_{-}}}.\]
	Using a result by \cite{kondo} on the characterization of the moduli space of smooth quartic curves as a ball complement quotient, we show that these two groups coincide. Analogous results have been used to study the monodromy and fundamental group of cubic surfaces (see, e.g. \cite{libgober}). Furthermore, as explained in \Cref{L+L-}, the lattices $L_{\pm}$ are pairwise orthogonal and therefore primitive within $\cohomoldeg{2}{X}{\zz}$. By virtue of the relation between the actions of $\rho_{4}$ on the discriminant groups of $L_{\pm}$, we obtain that if $\rho_{4}$ acts trivially on $L_{-}$ then it follows that the action on $L_{+}$ is trivial too. This implies that the monodromy map $\rho_{4}$ is completely determined by its action on the sublattice $L_{-}$, namely:
	
	\begin{thm}\label{TheoremK3}
		The monodromy representation
		\[\rho_{4}:\pi_{1}\Par{\mathcal{U}_{4}}\to U\Par{h_{L_{-}}}\]
		of the universal $(4, 4)$-branched cover of $\pp^{2}$ is surjective. In particular, its monodromy  is arithmetic.
	\end{thm}
	
	\subsection{Acknowledgements}
	
	I would like to thank Benson Farb and Eduard Looijenga for their invaluable guidance, advice and support through the making of and comments on this paper, as well as to Benson for suggesting this problem. I'm grateful to Madhav Nori, Michael Artin, Trevor Hyde, Peter Huxford, Olga Medrano Mart\'in del Campo, Ishan Banerjee, Nathaniel Mayer for helpful conversations, and to the Jump Trading Mathlab Fund for providing a working space where most of said conversations could be held. Finally, I would like to thank Curtis McMullen, Igor Dolgachev and Daniel Allcock for their helpful comments to make this paper more readable.
	
	\section{Degree $2$ del Pezzo Surfaces}\label{delPezzo}
	
	\subsection{General facts}\label{GeneralFactsDelPezzo}\label{GeneralFacts}
	
	Degree $2$ del Pezzo surfaces are realized as blow-ups of $\pp^{2}$ at $7$ points in general position. Given a degree $2$ del Pezzo surface $\mathcal{P}$, its anticanonical linear system $|-K_{\mathcal{P}}|$ defines a rational map $p:\mathcal{P}\to \pp^{2}$ which has as branch locus a smooth quartic curve $V\Par{f}$. Such surface $\mathcal{P}$ contains exactly $56$ exceptional divisors corresponding to the pullbakcs of the $28$ bitangents of $V\Par{f}$. These $56$ divisors can be labeled as follows:
	
	Let $P_{1}, P_{2}, \ldots, P_{7}$ be the seven points at which $\pp^{2}$ is blown up to obtain $\mathcal{P}$ and let $L_{1}, L_{2}, \ldots,L_{7}$ be the exceptional divisors corresponding to the blow-ups at these points. Let $e_{i}$ be the Poincar\'e dual of each divisor $L_{i}$ and $e_{0}=p^{\ast}\Par{\text{PD}\BPar{H}}$ be the Poincar\'e dual of the pullback of the hyperplane class $\BPar{H}$ in $\pp^{2}$. Then the $56$ exceptional divisors contained in $\mathcal{P}$ have corresponding cohomology classes given by
	\begin{align*}
		\BPar{L_{i}}&=e_{i} &&1\leq i\leq 7\\
		\BPar{L_{i}^{\ast}}&=3e_{0}-e_{i}-(e_{1}+e_{2}+\cdots+e_{7}) \\
		\BPar{L_{i, j}}&=e_{0}-e_{i}-e_{j} &&1\leq i<j\leq 7\\
		\BPar{L_{i, j}^{\ast}}&=2e_{0}+e_{i}+e_{j}-(e_{1}+e_{2}+\cdots+e_{7})
	\end{align*}
	These $56$ classes span $\cohomoldeg{2}{\mathcal{P}}{\zz}$, whose intersection form with respect to the basis $\WBPar{e_{0}, e_{1}, \ldots, e_{7}}$ is
	\begin{align*}
		\Bracket{\cdot, \cdot}_{\mathcal{P}}:\cohomoldeg{2}{\mathcal{P}}{\zz}&\times \cohomoldeg{2}{\mathcal{P}}{\zz}\to \zz &
		\Par{\mathbf{a}, \mathbf{b}}&\mapsto \mathbf{a}^{T}\begin{pmatrix}
			1 & 0 \\
			0 & -I_{7} \\
		\end{pmatrix}\mathbf{b}.
	\end{align*}
	
	\subsection{The Gosset and Schl\"afli graphs}\label{GossetSchlafli}
	
	As explained in Section $8$ of \cite{dolgachev}, the classes of the $56$ exceptional divisors of $\mathcal{P}$ are in correspondence with the vertices of the Gosset graph, where two vertices share an edge if and only if the inner product of the corresponding cohomology classes is equal to $1$. Moreover, the induced subgraph given by the vertices adjacent to a given vertex form a graph isomorphic to the Schl\"{a}fli graph, which gives the intersection pattern of the $27$ lines contained in a smooth cubic surface. The automorphim group of the Schl\"afli and Gosset graphs are the Weyl groups $W\Par{E_{6}}$ and $W\Par{E_{7}}$ respectively. A depiction of both graphs is shown in \Cref{SchlafliGosset}.
	
	
	This phenomenon is fundamentally related to the following geometric construction in \cite{harris}. A degree $2$ del Pezzo surface $\mathcal{P}$  may be regarded as a blow-up of a cubic surface $S$ over a point $Q$ not contained in any of its $27$ lines. The projections of the strict transforms of the $27$ lines contained in $S$ and the exceptional divisor over $Q$ via $p:\mathcal{P}\to\pp^{2}$ are in correspondence with the $28$ bitangents of the quartic curve $V\Par{f}\subset \pp^{2}$. Moreover, the $56$ exceptional divisors contained in $\mathcal{P}$ lie over each bitangent of $V\Par{f}$ in pairs. Each pair of such exceptional divisors has inner product $2$ with respect to the intersection form in $\cohomoldeg{2}{\mathcal{P}}{\zz}$ and they are precisely the pairs
	\[\WBPar{\Par{\BPar{L_{i}}, \BPar{L_{i}^{\ast}}}\mid 1\leq i\leq 7}\quad\text{ and }\quad\WBPar{\Par{\BPar{L_{i, j}}, \BPar{L_{i, j}^{\ast}}}\mid 1\leq i<j\leq 7}.\]
	We call each pair of such lines \emph{dual} and we denote the dual of a line $L$ by $L^{\ast}$. Let $L_{Q}$ be the exceptional divisor over $Q$, $\mathcal{S}$ be the set of strict transforms of the $27$ lines in the cubic surface $S$, and $\mathcal{S}^{\ast}$ be the set of dual lines to those in $\mathcal{S}$. A direct computation shows that
	\[\Bracket{\BPar{L_{Q}}, \BPar{L}}_{\mathcal{P}}=\begin{cases}
		0 & L\in \mathcal{S}\\
		1 & L\in \mathcal{S}^{\ast}
	\end{cases}\]
	Moreover, the intersection pattern of the lines in $\mathcal{S}$ and in $\mathcal{S}^{\ast}$ is given by the Schl\"afli graph. Hence, for any line $L$ contained in $\mathcal{P}$ we may define the following sets of $27$ lines associated to $L$:
	\[\mathcal{S}_{L}=\WBPar{L'\subset \mathcal{P}\mid \Bracket{\BPar{L}, \BPar{L'}}_{\mathcal{P}}=0}\quad\quad\quad \mathcal{S}^{\ast}_{L}=\WBPar{L'\subset \mathcal{P}\mid \Bracket{\BPar{L}, \BPar{L'}}_{\mathcal{P}}=1}\]
	where $\mathcal{S}_{L^{\ast}}=\mathcal{S}^{\ast}_{L}$. With this notation, we also have that
	\[L\in \mathcal{S}_{L'}\iff L'\in \mathcal{S}_{L}\quad\text{ and }\quad L\in \mathcal{S}^{\ast}_{L'}\iff L'\in \mathcal{S}^{\ast}_{L}.\]
	
	\section{Computing $\text{Im}\left(\rho_{2}\right)$}\label{Rho2}
	
	In this section we prove \Cref{TheoremDelPezzo}, namely, we show that
	\[\text{Im}\Par{\rho_{2}: \pi_{1}\Par{\mathcal{U}_{4}}\to \text{Aut}\Par{\cohomoldeg{2}{\mathcal{P}}{\zz}}}\cong W\Par{E_{7}}.\]
	The proof consists of $3$ main steps: computing the stabilizer of a line $L$ in $\mathcal{P}$, restricting the $\text{Im}\Par{\rho_{2}}$ and showing that $\text{Im}\Par{\rho_{2}}$ acts transitively on th $56$ lines contained in $\mathcal{P}$.
	
	\subsection{Stabilizer of a line $L$}\label{StabLine}
	
	In \cite{harris}, Harris studies the monodromy groups of the $28$ bitangents of a smooth quartic and the $27$ lines in a smooth cubic surface are studied. It is a classical result of Klein and Jordan (see \cite{harris} for a proof) that the monodromy group of $27$ lines in a cubic surface $S$ is isomorphic to
	\[W\Par{E_{6}}\cong O^{-}_{6}\Par{\zz/2\zz}.\]
	The monodromy group of the $28$ bitangents of a smooth quartic curve $V\Par{f}$ is isomorphic to
	\[W\Par{E_{7}}^{+}\cong O_{6}\Par{\zz/2\zz}\cong \text{Sp}_{6}\Par{\zz/2\zz}\]
	and moreover, we have that $W\Par{E_{7}}\cong W\Par{E_{7}}^{+}\times \zz/2\zz$. It follows that an element of the monodromy group of the $28$ bitangents of a quartic curve that fixes one bitangent is determined by its monodromy action on the $27$ lines on $S$ corresponding to the remaining bitangents. It is then shown in \cite{harris} that automorphisms of the $27$ lines in $S$ preserving a set of $6$ of these lines generate $W\Par{E_{6}}$ and can be realized as the monodromy action given by a path in the parameter space of smooth quartic curves $\mathcal{U}_{4}$, implying the following key lemma.
	\begin{lem}[Harris, \cite{harris}, Section II.4]\label{StabBitangent}
		The stabilizer of any bitangent in the monodromy group of  the $28$ bitangents of a smooth quartic curve is isomorphic to $W\Par{E_{6}}$.
	\end{lem}
	Now we turn to the monodromy action of interest on the $56$ lines in $\mathcal{P}$, proving the following.
	\begin{prop}\label{StabDelPezzo}
		The stabilizer in $\text{Im}\Par{\rho_{2}}$ of any line in $\mathcal{P}$  is isomorphic to $W\Par{E_{6}}$.
	\end{prop}
	\begin{proof}
		Suppose $g\in \pi_{1}\Par{\mathcal{U}_{4}}$ lies in the stabilizer of a line $L$ in $\mathcal{P}$. Then $g$ must lie in the stabilizer of its dual line $L^{\ast}$. Moreover, since the intersection form is preserved by the monodromy action, the set $\mathcal{S}_{L}$  must be permuted, and its permutation completely determines that of $\mathcal{S}^{\ast}_{L}$, hence the action of $g$. The automorphism group of $\mathcal{S}_{L}$ is then isomorphic to the automorphism group of their underlying bitangents, and by \Cref{StabBitangent} this is precisely the stabilizer of the bitangent lying under $L$, which is isomorphic to $W\Par{E_{6}}$.
	\end{proof}
	
	\subsection{Restricting $\text{Im}\Par{\rho_{2}}$}\label{Restriction}
	
	We proceed now to reduce $\text{Im}\Par{\rho_{2}}$ to two possibilities with the aid of the following two lemmas.
	
	\begin{lem}\label{UpperBound}
		$\text{Im}\Par{\rho_{2}}$ is contained in $W\Par{E_{7}}$.
	\end{lem}
	\begin{proof}
		We have seen that the automorphism group of the Gosset graph is isomorphic to $W\Par{E_{7}}$ and the intersection pattern of the $56$ lines in $\mathcal{P}$ is given by this graph. Since the monodromy action preserves the intersection form in $\cohomoldeg{2}{\mathcal{P}}{\zz}$ and the classes of the $56$ lines in $\mathcal{P}$ generate $\cohomoldeg{2}{\mathcal{P}}{\zz}$, it follows that the monodromy action of any element in $\pi_{1}\Par{\mathcal{U}_{4}}$ is determined by an automorphism of the Gosset graph, thus giving the desired restriction.
	\end{proof}
	\begin{lem}\label{LowerBound}
		$\text{Im}\Par{\rho_{2}}$ contains $W\Par{E_{7}}^{+}$.
	\end{lem}
	\begin{proof}
		The action of $\text{Im}\Par{\rho_{2}}$ on the $56$ lines in $\mathcal{P}$ preserves the $28$ pairs of lines lying over each bitangent of $V\Par{f}$. Recall that the automorphism group of the $28$ bitangents of $V\Par{f}$ is isomorphic to $W\Par{E_{7}}^{+}$. As noted in \Cref{StabDelPezzo} above, any such automorphism is realized by an element of $\pi_{1}\Par{\mathcal{U}_{4}}$, which in turn induces an automorphism of the $28$ pairs of lines in $\mathcal{P}$. This implies the desired contention.
	\end{proof}
	
	Since $W\Par{E_{7}}^{+}$ is a subgroup of index $2$ in $W\Par{E_{7}}$, \Cref{UpperBound} and \Cref{LowerBound} imply the following:
	\begin{prop}\label{2Possibilities}
		$\text{Im}\Par{\rho_{2}}$ is isomorphic either to $W\Par{E_{7}}^{+}$ or to $W\Par{E_{7}}$.
	\end{prop}
	
	\subsection{Transitivity on the lines contained in $\mathcal{P}$}
	
	Finally, we see $\text{Im}\Par{\rho_{2}}$ acts transitively on the $56$ lines, which will allow us to conclude our proof.
	
	\begin{lem}\label{TransitivitySameIntersection}
		For any line $L$ contained in $\mathcal{P}$, the group $\text{Im}\Par{\rho_{2}}$ acts transitively on $\mathcal{S}_{L}$ and on $\mathcal{S}^{\ast}_{L}$.
	\end{lem}
	\begin{proof}
		By \Cref{StabDelPezzo}, the stabilizer of $L$ is isomorphic to $W\Par{E_{6}}$, which is the automorphism group of $\mathcal{S}_{L}$. Since the $W\Par{E_{6}}$ action on the  lines in a smooth cubic surface is transitive, so is the action on $\mathcal{S}_{L}$. Since $\text{Im}\Par{\rho_{2}}$ permutes pairs consisting of a line and its dual, transitivity on $\mathcal{S}^{\ast}_{L}$ follows from that on $\mathcal{S}_{L}$.
	\end{proof}
	
	\begin{lem}\label{TransitivityDifferentIntersection}
		The $\text{Im}\Par{\rho_{2}}$-orbit of a line $L$ in $\mathcal{P}$ contains lines in both $\mathcal{S}_{L}$ and $\mathcal{S}^{\ast}_{L}$.
	\end{lem}
	\begin{proof}
		Consider a line $N\neq L, L^{\ast}$ in the orbit of $L$, and suppose $N\in \mathcal{S}^{\ast}_{L}$. The intersection pattern of the lines in $\mathcal{S}^{\ast}_{N}$ is given by the Schl\"afli graph, where each vertex has valence $16$. Since $L\in \mathcal{S}^{\ast}_{N}$, this implies that $\mathcal{S}^{\ast}_{N}$ intersects both sets $\mathcal{S}_{L}$ and $\mathcal{S}^{\ast}_{L}$. By \Cref{TransitivitySameIntersection}, the stabilizer of $N$ acts transitively on $\mathcal{S}^{\ast}_{N}$, and thus the result follows. In the case $N\in \mathcal{S}_{L}$, the proof is analogous.
	\end{proof}
	
	\begin{prop}\label{Transitivity}
		$\text{Im}\Par{\rho_{2}}$ acts transitively on the $56$ lines of $\mathcal{P}$.
	\end{prop}
	
	\begin{proof}
		For a line $L$ in $\mathcal{P}$, the orbit-stabilizer theorem implies that
		\[\left|\text{Im}\Par{\rho}\right|=\left|\text{Orbit}\Par{L}\right|\cdot\left|\text{Stab}\Par{L}\right|.\]
		\Cref{StabDelPezzo} and \Cref{2Possibilities} imply that
		\begin{align*}
			\Abs{\text{Stab}\Par{L}}&=\Abs{W\Par{E_{6}}}=51840 \\
			\Abs{\text{Im}\Par{\rho}}&\in
			\WBPar{\Abs{W\Par{E_{7}}}, |W\Par{E_{7}}^{+}|}=
			\WBPar{2903040, 1451520}
		\end{align*}
		so it follows that $\Abs{\text{Orbit}\Par{L}}\in \WBPar{28, 56}$. By \Cref{TransitivityDifferentIntersection}, the orbit of $L$ intersects both $\mathcal{S}_{L}$ and $\mathcal{S}^{\ast}_{L}$ and by \Cref{TransitivitySameIntersection}, the $\text{Im}\Par{\rho_{2}}$-action is transitive on each of $\mathcal{S}_{L}$ and $\mathcal{S}^{\ast}_{L}$, implying that
		\[\Abs{\text{Orbit}\Par{L}}\geq 1+27+27>28\]
		and therefore $\Abs{\text{Orbit}\Par{L}}=56$. This implies that $\text{Im}\Par{\rho_{2}}\cong W\Par{E_{7}}$.
	\end{proof}
	
	\section{Lattices}\label{Lattices}
	
	In this section we dwelve into the lattice theory of associated to the $K3$ surface. We will restrict $\rho_{4}$ to two sublattices $L_{+}$ and $L_{-}$ of $\cohomoldeg{2}{X}{\zz}$ associated to the cyclic $\zz/4\zz$ deck group action on $X$. In doing so, our goal is to relate these two restrcitions via the commutative diagram given in \Cref{commutativeDiagram}.
	
	\subsection{Preliminaries}\label{LatticePrelim}
	
	We recall general facts and definitions about lattices which will be used in the treatement of the monodromy of the family of $K3$ surfaces of interest in this paper. This exposition is contained in chapter 2 of \cite{kondoBook}.
	
	A \emph{lattice} $\Par{L, \Bracket{\,,}_{L}}$ of rank $r$ is a pair consisting of a free abelian group $L$ of rank $r$ together with a symmetric bilinear form $\Bracket{\,,}_{L}:L\times L\to \zz$. The quotient 
	\[A_{L}=L^{\ast}/L\]
	denoted as the \emph{discriminant group} of $L$. If $A_{L}$ is trivial, we say that $L$ is \emph{unimodular}.
	
	A sublattice $S\subset L$ is \emph{primitive} if $L/S$ is torsion-free. If $L$ is non-degenerate, the orthogonal complement $S^{\perp}$ of any sublattice $S\subset L$ is primitive.
	
	The group of automorphisms of $L$ is denoted by $O\Par{L}$ and is called its \emph{orthogonal group}. For an even lattice $L$, we define its \emph{discriminant quadratic form} by
	\[q_{L}:A_{L}\to \qq/2\zz\quad\quad\quad q_{L}\Par{x}=\Bracket{x, x}_{L\otimes \qq}\]
	where $\Bracket{\,,}_{L\otimes \qq}$ is a bilinear extension of $\Bracket{\,,}_{L}$ to $L\otimes \qq$. A subgroup $H\subset A_{L}$ is called \emph{isotropic} if $q_{L}\vert_{H}=0$. We denote the group of group automorphisms of $A_{L}$ preserving $q_{L}$ by  $O\Par{q_{L}}$. Any automorphism of $L$ can be extended to $L^{\ast}$, thus inducing an automorphism of $A_{L}$ which respects $q_{L}$. This induces a natural group homomorphism
	\[O\Par{L}\to O\Par{q_{L}}.\]
	A lattice $L$ is called \emph{$2$-elementary} if $A_{L}\cong \Par{\zz/2\zz}^{\oplus l}$ for some $l$. If $L$ is an indefinite $2$-elementary lattice, then the homomorphism $O\Par{L}\to O\Par{q_{L}}$ is known to be surjective.
	
	Let $L$ be an even unimodular lattice, $S\subset L$ a primitive sublattice and $T=S^{\perp}$ its orthogonal complement. There are natural embeddings
	\[S\oplus T\subset L\cong L^{\ast}\subset S^{\ast}\oplus T^{\ast}.\]
	Then $L/\Par{S\oplus T}$ is embedded as an isotropic subgroup of $A_{S}\oplus A_{T}$. This gives the commutative diagram
	
	\begin{equation}\label{discriminantDiagram}
		\begin{tikzcd}
			& & A_{S}\arrow[dd, "\gamma_{ST}"]\\
			L/\Par{S\oplus T}\arrow[r, hookrightarrow, "\iota"]\arrow[drr, bend right=15, "\cong" ']\arrow[urr, bend left=15, "\cong"]& A_{S}\oplus A_{T}\arrow[ur, twoheadrightarrow, "p_{S}" ']\arrow[dr, twoheadrightarrow, "p_{T}"] & \\
			& & A_{T}
		\end{tikzcd}\tag{$\ast$}
	\end{equation}
	where the maps $p_{S}\circ \iota$ and $p_{T}\circ \iota$ induce group isomorphims between $L/\Par{S\oplus T}$ and $A_{S}, A_{T}$. In particular, there is an isomorphism $A_{S}\cong A_{T}$. Furthermore, the isomorphism
	\[\gamma_{ST}=p_{T}\circ \Par{p_{S}\vert_{\iota\Par{L/\Par{S\oplus T}}}}^{-1}:A_{S}\to A_{T}\]
	satisfies that $q_{S}=-\gamma_{ST}\circ q_{T}$, hence $O\Par{q_{S}}\cong O\Par{q_{T}}$ via the map $c_{\gamma_{ST}}$ given by conjugation by $\gamma_{ST}$. Let $c_{S}$ and $c_{T}$ be the conjugation map given by
	\begin{align*}
		c_{S}:O\Par{q_{S}}&\to \text{Aut}\Par{L/\Par{S\oplus T}} & c_{T}:O\Par{q_{T}}&\to \text{Aut}\Par{L/\Par{S\oplus T}}\\
		\varphi&\mapsto  \Par{p_{S}\circ \iota}^{-1}\circ \varphi \circ \Par{p_{S}\circ \iota} & \varphi&\mapsto  \Par{p_{T}\circ \iota}^{-1}\circ \varphi \circ \Par{p_{T}\circ \iota}
	\end{align*}
	By commutativity of the diagram (\ref{discriminantDiagram}), the images of $c_{S}$ and $c_{T}$ coincide, and we denote $O\Par{L/\Par{S\oplus T}}$ by this image. Therefore, the following commutative diagram is induced:
	\begin{equation}\label{orthogonalDiagram}
		\begin{tikzcd}
			& O\Par{q_{S}}\arrow[dd, "c_{\gamma_{ST}}", "\cong" ']\\
			O\Par{L/\Par{S\oplus T}}
			\arrow[ur, "\cong", "c_{S}" ']
			\arrow[dr, "c_{T}", "\cong" '] & \\
			& O\Par{q_{T}}
		\end{tikzcd}\tag{$\ast \ast$}
	\end{equation}
	
	\subsection{$K3$ lattices}\label{L+L-}
	
	A $K3$ surface $X$ is a simply connected complex surface whose canonical class $K_{X}$ is trivial. The \emph{$K3$-lattice} is the lattice given by $\cohomoldeg{2}{X}{\zz}$ together with its intersection pairing given by the cup product. This lattice is isomorphic to the even unimodular lattice of signature type $\Par{3, 19}$:
	\[U^{\oplus 3}\oplus E_{8}\Par{-1}^{\oplus 2}\]
	where $U$ is the hyperbolic lattice and $E_{8}\Par{-1}$ is the negative $E_{8}$ lattice. Namely,
	\[U=\begin{pmatrix}0&1 \\ 1&0\end{pmatrix}\quad E_{8}\Par{-1}=\begin{pmatrix}
		-2 & 1 & & & & & & \\
		1 & -2 & 1 & & & & & \\
		& 1 & -2 & 1& & & & \\
		& & 1 & -2 & 1 & & & \\
		& & & 1 & -2 & 1 & & 1 \\
		& & & & 1 & -2 & 1 & \\
		& & & & & 1 & -2  & \\
		& & & & 1 & & & -2 
	\end{pmatrix}\]
	
	We now turn to our family of $K3$ surfaces $X$ branching over a smooth quartic curve. Before computing the monodromy of this family, we introduce two auxiliary sublattices of $\cohomoldeg{2}{X}{\zz}$ crucial to its computation. These lattices arise from the $\zz/4\zz$ deck group action on $X$, and are extensively used in \cite{kondo} in order to describe the moduli space of smooth quartic curves in $\pp^{2}$ as a ball complement quotient. The deck group action on $X$ induces an action $T: \cohomoldeg{2}{X}{\zz}\to \cohomoldeg{2}{X}{\zz}$ which satisfies $T^{4}=I$ and commutes with the action of $\rho_{4}$. Two sublattices associated to $T$ that arise naturally are
	\[L_{+}=\left\{x\in \cohomoldeg{2}{X}{\zz}\mid T^{2}x=x\right\}\quad\quad\quad L_{-}=\left\{x\in \cohomoldeg{2}{X}{\zz}\mid T^{2}x=-x\right\}.\]
	
	\begin{prop}\label{orthogonal}
		The lattices $L_{+}$ and $L_{-}$ are the orthogonal complement of each other in $\cohomoldeg{2}{X}{\zz}$.
	\end{prop}
	
	\begin{proof}
		We will show that $L_{+}^{\perp}=L_{-}$ only, as the proof that $L_{-}^{\perp}=L_{+}$ is analogous. For $x\in L_{+}$ and $y\in L_{-}$,
		\[\Bracket{x, y}=\Bracket{T^{2}x, T^{2}y}=\Bracket{x, -y}=-\Bracket{x, y}\]
		and thus $\Bracket{x, y}=0$. This shows that $L_{-}\subset L_{+}^{\perp}$. To see the reverse inclusion, suppose $y\in L_{+}^{\perp}$. Then
		\[\Bracket{x, \Par{T^{2}+I}y}=0\quad \forall x\in L_{+}\]
		We have that $\Par{T^{2}+I}y\in L_{+}$ since $T^{4}=I$. Since $L_{+}$ is non-degenerate, $\Par{T^{2}+ I}y=0$, so $y\in L_{-}$.
	\end{proof}
	
	\Cref{orthogonal} implies that $L_{+}$ and $L_{-}$ are primitive lattices in $\cohomoldeg{2}{X}{\zz}$. Let $\mathcal{P}$ be the intermediate double branched cover of $\pp^{2}$ corresponding to $X$. Then $X$ is also a double branched cover of $\mathcal{P}$, and the branched cover map $p:X\to \mathcal{P}$ induces the pullback
	\begin{align*}
		p^{\ast}:\cohomoldeg{2}{\mathcal{P}}{\zz}&\to \cohomoldeg{2}{X}{\zz} \\
		e_{i}&\mapsto \tilde{e_{i}}
	\end{align*}
	The lattice $L_{+}$ is generated by $\tilde{e_{0}}, \tilde{e_{1}}, \ldots, \tilde{e_{7}}$. Letting $H=\cohomoldeg{2}{X}{\zz}/\Par{L_{+}\oplus L_{-}}$, we have that
	\[H\cong A_{L_{+}}\cong A_{L_{-}}\cong \Par{\zz/2\zz}^{\oplus 8}\]
	and thus $L_{+}$ and $L_{-}$ are $2$-elementary even lattices. In \cite{nikulin}, Theorem 3.6.2, $2$-elementary even lattices are classified by their rank and minimal number of generators, hence
	\[L_{+}\cong \Bracket{2}\oplus A_{1}^{\oplus 7}\quad\quad\quad L_{-}\cong A_{1}^{\oplus 2}\oplus D_{4}^{\oplus 2}\oplus U\oplus U\Par{2}.\]
	In \cite{kondo} it is shown that $O\Par{q_{L_{+}}}$ is isomorphic to a split extension of $\text{Sp}_{6}\Par{\zz/2\zz}$ by $\zz/2\zz$, thus a semidirect product. Since $\text{Aut}\Par{\zz/2\zz}$ is trivial, it  is a direct product so $O\Par{q_{+}}\cong W\Par{E_{7}}$. Diagram (\ref{orthogonalDiagram}) then gives
	\[O\Par{q_{L_{+}}}\cong O\Par{q_{L_{-}}}\cong  W\Par{E_{7}}.\]
	
	\subsection{A useful diagram}\label{MasterDiagram}
	
	In order to relate $\rho_{4}$ to the lattices $L_{+}, L_{-}$, we will use a commutative diagram consisting of quotient and restriction maps. The key observation lies in the fact that $\rho_{4}$ acts on each summand $L_{+}$ and $L_{-}$.
	
	\begin{prop}\label{restriction}
		The action of $\rho_{4}$ can be restricted from $\cohomoldeg{2}{X}{\zz}$ to $L_{+}$ and to $L_{-}$.
	\end{prop}
	
	\begin{proof}
		Since the action of $\rho_{4}$ and $T$ commute, we have
		\[T^{2}\rho_{4}\Par{\gamma}x=\rho_{4}\Par{\gamma}T^{2}x=\pm\rho_{4}\Par{\gamma}x.\]
		for $x\in L_{\pm}$ and $\gamma\in \pi_{1}\Par{\mathcal{U}_{4}}$, and hence $\rho_{4}\Par{\gamma}x\in L_{\pm}$.
	\end{proof}
	
	We proceed to define the necessary maps for our diagram. \Cref{restriction} tells us that there are well defined restriction maps $\text{res}\Par{L_{+}}$ and $\text{res}\Par{L_{-}}$ given by
	\begin{align*}
		\text{res}\Par{L_{+}}:\text{Im}\Par{\rho_{4}}&\to O\Par{L_{+}} & \text{res}\Par{L_{-}}:\text{Im}\Par{\rho_{4}}&\to O\Par{L_{-}} \\
		\rho_{4}\Par{\gamma} &\mapsto \rho_{4}\Par{\gamma}\vert_{L_{+}} & \rho_{4}\Par{\gamma} &\mapsto \rho_{4}\Par{\gamma}\vert_{L_{-}} .
	\end{align*}
	We define the composition maps $\rho_{4}^{+}=\text{res}\Par{L_{+}}\circ\rho_{4}$ and $\rho_{4}^{-}=\text{res}\Par{L_{-}}\circ\rho_{4}$. This gives the map
	\begin{align*}
		\text{mod}\Par{L_{+}\oplus L_{-}}:\text{Im}\Par{\rho_{4}} &\to O\Par{H}\\
		\rho_{4}\Par{\gamma}&\mapsto \left.\Par{\widetilde{\rho_{4}^{+}\Par{\gamma}} \pmod{L_{+}}, \widetilde{\rho_{4}^{-}\Par{\gamma}}\pmod{L_{-}}}\right|_{H}
	\end{align*}	 
	which is well defined regarding $H$ as a subgroup of $A_{L_{+}}\oplus A_{L_{-}}$. Finally, since $L_{+}$ and $L_{-}$ are even indefinite $2$-elementary lattices, we have the surjective homomorphisms
	\begin{align*}
		\text{mod}\Par{L_{+}}:O\Par{L_{+}}&\to O\Par{q_{L_{+}}} & \text{mod}\Par{L_{-}}:O\Par{L_{-}}&\to O\Par{q_{L_{-}}} \\
		\varphi &\mapsto \tilde{\varphi}\pmod{L_{+}} & \varphi &\mapsto \tilde{\varphi}\pmod{L_{-}}
	\end{align*}
	
	Before putting these maps together, we compute $\text{Im}\Par{\rho_{4}^{+}}$.
	
	\begin{prop}\label{Rho4+}
		$\text{Im}\Par{\rho_{4}^{+}}\cong W\Par{E_{7}}$.
	\end{prop}
	
	\begin{proof}
		The action of $\rho_{2}$ on $\cohomoldeg{2}{\mathcal{P}}{\zz}$ completely determines that of $\rho_{4}$ on $L_{+}$ by conjugation $c_{p^{\ast}}$ with the pullback $p^{\ast}:\cohomoldeg{2}{\mathcal{P}}{\zz}\to L_{+}$. Together with \Cref{TheoremDelPezzo}, this gives the following commutative diagram:
		
		\begin{center}
			\begin{tikzcd}
				\pi_{1}\Par{\mathcal{U}_{4}}\arrow[r, twoheadrightarrow, "\rho_{4}"] \arrow[d, twoheadrightarrow, "\rho_{2}" '] \arrow[dr, twoheadrightarrow, "\rho_{4}^{+}"]& \text{Im}\Par{\rho_{4}}\arrow[d, twoheadrightarrow, "\text{res}\Par{L_{+}}"]\\
				W\Par{E_{7}}\arrow[r, "c_{p^{\ast}}" '] & \text{Im}\Par{\rho_{4}^{+}} 
			\end{tikzcd}
		\end{center}
		Since $\rho_{4}^{+}$ is surjective, so is $c_{p^{\ast}}$. Since the pullback $p^{\ast}$ is a group isomorphism between $\cohomoldeg{2}{\mathcal{P}}{\zz}$ and $L_{+}$, we have that $c_{p^{\ast}}$ is injective. Hence, $c_{p^{\ast}}$ is an isomorphism and $\text{Im}\Par{\rho_{4}^{+}}\cong W\Par{E_{7}}$.
	\end{proof}
	
	The restriction of $\text{mod}\Par{L_{+}}$ to $\text{Im}\Par{\rho_{4}^{+}}$ is surjective and is thus an isomorphism since $O\Par{q_{L_{+}}}\cong W\Par{E_{7}}$. Putting together the discussion above, we obtain the following.
	
	\begin{prop}\label{commutativeDiagram}
		The following diagram commutes:
		\begin{center}
			\begin{tikzcd}
				& & O\Par{L_{-}}
				\arrow[rr, twoheadrightarrow, "\text{mod}\Par{L_{-}}"] && O\Par{q_{L_{-}}} \\
				\pi_{1}\Par{\mathcal{U}_{4}}
				\arrow[r, twoheadrightarrow, "\rho_{4}"]
				\arrow[drr, twoheadrightarrow, bend right=15, "\rho_{4}^{+}" ']
				\arrow[urr, bend left=15, "\rho_{4}^{-}"]& \text{Im}\Par{\rho_{4}}
				\arrow[rr, "\text{mod}\Par{L_{+}\oplus L_{-}}"]
				\arrow[dr, twoheadrightarrow, "\text{res}\Par{L_{+}}" ']
				\arrow[ur, "\text{res}\Par{L_{-}}"] &&O\Par{H}
				\arrow[ur, "c_{L_{-}}", "\cong" ']
				\arrow[dr, "c_{L_{+}}" ', "\cong"]&  \\
				& & W\Par{E_{7}}
				\arrow[rr, twoheadrightarrow, "\cong", "\text{mod}\Par{L_{+}}" '] && O\Par{q_{L_{+}}}
				\arrow[uu, "\cong", "c_{\gamma_{L_{+}L_{-}}}" ']
			\end{tikzcd}
		\end{center}
		
	\end{prop}
	
	\section{Computing $\text{Im}\left(\rho_{4}\right)$}\label{Rho4}
	In this section we prove \Cref{TheoremK3}. To do so, we first show that $\text{Im}\Par{\rho_{4}}$ and $\text{Im}\Par{\rho_{4}^{-}}$ are isomorphic. We then proceed to study $\zz\BPar{i}$-lattice structure on $L_{-}$ to describe $\text{Im}\Par{\rho_{4}^{-}}$. Finally, we use Kondo's description of the moduli space of smooth quartic curves proved in \cite{kondo} in order to compute $\text{Im}\Par{\rho_{4}}$.
	\subsection{Reduction to $\rho_{4}^{-}$}\label{reductionRho4-}
	
	\begin{prop}\label{reduction}
		$\text{Im}\Par{\rho_{4}}\cong \text{Im}\Par{\rho_{4}^{-}}$.
	\end{prop}
	
	\begin{proof}
		First, we show that for any $\gamma\in \pi_{1}\Par{\mathcal{U}_{4}}$, if $\rho_{4}^{-}\Par{\gamma}$ is trivial, so is $\rho_{4}^{+}\Par{\gamma}$. Using the diagram in \Cref{commutativeDiagram}, we have that
		\begin{align*}
			\rho_{4}^{-}\Par{\gamma}=0&\implies \text{mod}\Par{L_{-}}\circ \rho_{4}^{-}\Par{\gamma} = 0 \\
			&\implies c_{\gamma_{L_{+}L_{-}}}\circ \text{mod}\Par{L_{+}}\circ \rho_{4}^{+}\Par{\gamma} = 0 \\
			(c_{\gamma_{L_{+}L_{-}}}\text{ is injective}) & \implies \text{mod}\Par{L_{+}}\circ \rho_{4}^{+}\Par{\gamma} = 0 \\
			(\text{mod}\Par{L_{+}}\text{ is injective}) & \implies \rho_{4}^{+}\Par{\gamma} = 0.
		\end{align*}
		Finally, we show that the map $\text{res}\Par{L_{-}}$ is injective. Let $\rho_{4}\Par{\gamma}\in \ker\Par{\text{mod}\Par{L_{-}}}$ for some $\gamma\in \pi_{1}\Par{\mathcal{U}_{4}}$. Then
		\[\rho_{4}^{-}\Par{\gamma}=0\implies \rho_{4}^{+}\Par{\gamma}=0.\]
		Hence, the action of $\rho_{4}\Par{\gamma}$ on $L_{+}\oplus L_{-}$ is trivial. Since $L_{+}\oplus L_{-}$ is a finite index sublattice of $\cohomoldeg{2}{X}{\zz}$, $\rho_{4}\Par{\gamma}$ acts trivially on $\cohomoldeg{2}{X}{\zz}$, so $\rho_{4}\Par{\gamma}=0$, implying the desired injectivity. Altogether, $\text{res}\Par{L_{-}}$ is an isomorphism onto its image and the proposition follows.
	\end{proof}
	
	
	\subsection{$\mathbb{Z}\left[i\right]$-module structure on $L_{-}$}\label{ZiModule}
	
	The action $T$ endows $L_{-}$ with a $\zz\BPar{i}$-module structure since $T^{2}+I$ acts trivially on $L_{-}$. Furthermore, the Hodge decomposition on $\cohomoldeg{2}{X}{\cc}$ provides $L_{-}$ a hermitian form induced by its intersection form. We study this structure on $L_{-}$ following the technique in \cite{ACT}.
	
	The action $T$ induces a decomposition
	\[\cohomoldeg{2}{X}{\cc}\cong\bigoplus_{\zeta^{4}=1}\cohomoldeg{2}{X}{\cc}_{\zeta}\cong \bigoplus_{\zeta^{4}=1}V_{\zeta}\]
	where $V_{\zeta}=\cohomoldeg{2}{X}{\cc}_{\zeta}:=\ker\Par{T-\zeta I}$. The eigenspaces $V_{i}, V_{-i}$ are conjugate and $L_{-}\otimes_{\zz}\cc\cong V_{i}\oplus V_{-i}$. Let $j:L_{-}\to V_{i}$ be the $\zz\BPar{i}$-linear composition map given by
	
	\begin{center}
		\begin{tikzcd}
			& \cohomoldeg{2}{X}{\cc}\arrow[dr, twoheadrightarrow, "\text{proj}_{V_{i}}"] & \\
			L_{-}\arrow[ur, hookrightarrow, "\nu"]\arrow[rr, "j"]& & V_{i}
		\end{tikzcd}
	\end{center}
	
	where $\nu:\cohomoldeg{2}{X}{\zz}\hookrightarrow \cohomoldeg{2}{X}{\cc}$ is the natural inclusion and $\text{proj}_{V_{i}}: \cohomoldeg{2}{X}{\cc}\to V_{i}$ is a projection. The bilinear form on $L_{-}$ can be extended $\zz\BPar{i}$-linearly to $L_{-}\otimes_{\zz\BPar{i}} \cc$ and $V_{i}$ can be given a hermitian form induced by the intersection pairing, given by $h\Par{a, b}=2\Bracket{a, \overline{b}}$.
	
	\begin{prop}\label{ViIsometricEmbedding}
		The map $j_{\cc}:L_{-}\otimes_{\zz\BPar{i}} \cc\to V_{i}$ is an isometric isomorphism. 
	\end{prop}
	
	\begin{proof}
		Let $a\in L_{-}$. Since $V_{i}, V_{-i}$ are conjugate, $a = j\Par{a}+\overline{j\Par{a}}$. We then have that
		\begin{align*}
			\Bracket{a, a}&=\Bracket{j\Par{a}+\overline{j\Par{a}}, j\Par{a}+\overline{j\Par{a}}} \\
			&= \Bracket{j\Par{a}, j\Par{a}}+\Bracket{\overline{j\Par{a}}, \overline{j\Par{a}}}+2\Bracket{j\Par{a}, \overline{j\Par{a}}} \\
			h\Par{j\Par{a}, j\Par{a}} &=2\Bracket{j\Par{a}, \overline{j\Par{a}}}
		\end{align*}
		hence $j_{\cc}$ is an isometry. Since $\dim_{\cc}V_{i}=\dim_{\cc}L_{-}\otimes_{\zz\BPar{i}}\cc=7$, it follows that $j_{\cc}$ is an isomorphism.
	\end{proof}
	
	\subsection{The Hodge structure on $\cohomoldeg{2}{X}{\cc}$}
	
	The Hodge structure on $\cohomoldeg{2}{X}{\cc}$ is given by
	\begin{align*}
		\cohomoldeg{2}{X}{\cc}&\cong \cohomoldeg{2, 0}{X}{\cc}\oplus \cohomoldeg{1, 1}{X}{\cc}\oplus \cohomoldeg{0, 2}{X}{\cc} \\
		&\cong \cc\oplus \cc^{20}\oplus \cc
	\end{align*}
	and it is determined by the line $H^{2, 0}$ since $H^{0, 2}\cong \overline{H^{2, 0}}$ and $H^{1, 1}\cong \Par{H^{2, 0}\oplus H^{0, 2}}^{\perp}$. Since $T$ is of finite order, it is a biholomorphism and therefore respects the Hodge decomposition on $\cohomoldeg{2}{X}{\cc}$. This further decomposes the spaces $V_{\zeta}$ into mutually orthogonal subspaces
	\[V_{\zeta}\cong V_{\zeta}^{2, 0}\oplus V_{\zeta}^{1, 1}\oplus V_{\zeta}^{0, 2}.\]
	with respect to the intersection pairing on $\cohomoldeg{2}{X}{\cc}$, given by
	\[\Bracket{\alpha, \overline{\beta}}=\int_{X}\alpha\wedge\overline{\beta}.\]
	Let $h'=\Abs{z_{0}}^{2}-\Abs{z_{1}}^{2}-\cdots-\Abs{z_{6}}^{2}$ be the standard quadratic form of signature $\Par{1, 6}$.
	
	\begin{prop}\label{ViSignature}
		The hermitian vector space $\Par{V_{i}, h}$ is isomorphic to $\cc^{1, 6}=\Par{\cc^{7}, h'}$.
	\end{prop}
	
	\begin{proof}
		We will show that $h$ is positive-definite on $V_{i}^{2, 0}\oplus V_{i}^{0, 2}$ and negative-definite on $V_{i}^{1, 1}$. We will also show the former and latter spaces are $1$ and $6$ dimensional respectively, concluding the proof.
		
		Let $\omega\in \cohomoldeg{2, 0}{X}{\cc}$ be a holomorphic form. Then $T\omega$ must be holomorphic too. Since $V_{1}\oplus V_{-1}$ is spanned by rational clases, $V_{1}\oplus V_{-1}\subset \cohomoldeg{1, 1}{X}{\cc}$ and hence $T\omega\neq \pm \omega$. This implies that $T\omega = \pm i\omega$ and therefore $T\overline{\omega}=\mp i\overline{\omega}$, hence
		\[V_{i}^{2, 0}\oplus V_{i}^{0, 2}\cong V_{-i}^{2, 0}\oplus V_{-i}^{0, 2}\cong \cc.\]
		Moreover, the intersection pairing on $\omega$ satisfies
		\[h\Par{\omega, \omega}=2\Bracket{\omega, \overline{\omega}}=2\int_{X}\omega\wedge\overline{\omega}>0\]
		so $V_{i}^{2, 0}\oplus V_{i}^{0, 2}$ is positive-definite. The K\"ahler form $\kappa$ associated to $X$ lies in $\cohomoldeg{1, 1}{X}{\cc}$ and has positive intersection pairing with itself being an integral cohomology class, so $\Bracket{\kappa, \overline{\kappa}}>0$. Since $\cohomoldeg{2}{X}{\cc}$ has signature $\Par{3, 19}$, the classes $\omega, \overline{\omega}, \kappa$ span a maximal positive-definite subspace of $\cohomoldeg{2}{X}{\cc}$. This implies that $V_{i}^{1, 1}$ and $V_{-i}^{1, 1}$ are negative-definite and
		\[V_{i}^{1, 1}\cong V_{-i}^{1, 1}\cong \cc^{6}.\]
	\end{proof}
	
	\subsection{$\zz\BPar{i}$-lattice structure on $L_{-}$}
	
	We now turn to describe $L_{-}$ as a $\zz\BPar{i}$-lattice. A $\zz\BPar{i}$-lattice structure is induced on $L_{-}$ via the action $T$ by endowing $L_{-}$ with the hermitian form
	\[\Bracket{x, y}_{\zz\BPar{i}}=\Bracket{x, y}-i\Bracket{x, Ty}.\]
	
	\begin{prop}\label{ZiModuleOnL-}
		As a $\zz\BPar{i}$-lattice, $L_{-}$ is isomorphic to $\zz\BPar{i}^{7}$ equipped with the type $\Par{1, 6}$ quadratic form
		\[h_{L_{-}}=-2\Par{\Abs{z_{0}}^{2}
			+\Abs{z_{1}}^{2}
			+\Abs{z_{2}}^{2}
			+\Abs{z_{3}}^{2}
			+\Abs{z_{4}}^{2}
			-\Re\Par{{z_{1}\overline{z_{2}}}+{z_{3}\overline{z_{4}}}+{z_{5}\overline{z_{6}}}}
			-\Im\Par{{z_{1}\overline{z_{2}}}+{z_{3}\overline{z_{4}}}+{z_{5}\overline{z_{6}}}}}.\]
	\end{prop}
	
	\begin{proof}
		In Chapter 10 of \cite{kondoBook}, the action of $T$ on
		\[L_{-}\cong A_{1}^{\oplus 2}\oplus D_{4}^{\oplus 2}\oplus U\oplus U\Par{2}\]
		is described. This action is diagonal, acting on $A_{1}^{\oplus 2}$, $U\oplus U\Par{2}$ and each copy of $D_{4}$ by blocks. This implies that each of these summands carries a $\zz\BPar{i}$-lattice structure, and in order to determine the quadratic form on $L_{-}$, it suffices to determine the quadratic form on each summand separately.
		\begin{itemize}
			\item $A_{1}^{\oplus 2}$ : As a $\zz$-lattice, it is generated by $u, v$ and $T$ acts by $Tu= v$. Hence, $u$ generates $A_{1}^{\oplus 2}$ as a $\zz\BPar{i}$-module. Since $\Bracket{u, u}_{\zz\BPar{i}}=-2$, the quadratic form on $A_{1}^{\oplus 2}$ is given by the $\Par{0, 1}$ form
			\[h_{A_{1}^{\oplus 2}}=\Bracket{z, z}_{\zz\BPar{i}}=-2\Abs{z}^{2}.\]
			
			\item $U\oplus U\Par{2}$ : As a $\zz$-lattice, it is generated by $e, f, e', f'$ where $e, f$ generate $U$ and $e', f'$ generate $U\Par{2}$. The action of $T$ is given by 
			\[
			Te=-e-e'\quad\quad\quad Tf=f-f'
			\]
			and thus $e, f$ generate $U\oplus U\Par{2}$ as a $\zz\BPar{i}$-module. We then have
			\[\begin{pmatrix}
				\Bracket{e, e}_{\zz\BPar{i}} & \Bracket{e, f}_{\zz\BPar{i}} \\
				\Bracket{f, e}_{\zz\BPar{i}} & \Bracket{f, f}_{\zz\BPar{i}}
			\end{pmatrix}=
			\begin{pmatrix}
				0 & 1-i \\
				1+i & 0
			\end{pmatrix}
			\]
			hence the quadratic form on $U\oplus U\Par{2}$ is given by the $\Par{1, 1}$ form
			\[h_{U\oplus U\Par{2}}=\Bracket{\Par{z, w}, \Par{z, w}}_{\zz\BPar{i}}=2\Re\Par{z\overline{w}}+2\Im\Par{z\overline{w}}.\]
			
			\item $D_{4}$ : As a $\zz$-lattice, $D_{4}$ is isomorphic to
			\[\WBPar{\Par{x_{1}, x_{2}, x_{3}, x_{4}}\subset \zz^{4}\mid x_{1}+x_{2}+x_{3}+x_{4}\equiv 0\pmod{2}}.\]
			equipped with the negative dot product. The basis $\WBPar{\Par{1, 1, 0, 0}, \Par{-1, 1, 0, 0}, \Par{0, -1, 1, 0}, \Par{0, 0, -1, 1}}$ gives the usual Cartan matrix describing the $\zz$-lattice structure of $D_{4}$. The action of $T$ is given by
			\[T\Par{x_{1}, x_{2}, x_{3}, x_{4}}=\Par{x_{2}, -x_{1}, x_{4}, -x_{3}}.\]
			Altogether, $p = (1, 1, 0, 0)$ and $q = (0, -1, 1, 0)$ generate $D_{4}$ as a $\zz\BPar{i}$-module. We then have
			\[\begin{pmatrix}
				\Bracket{p, p}_{\zz\BPar{i}} & \Bracket{p, q}_{\zz\BPar{i}} \\
				\Bracket{q, p}_{\zz\BPar{i}} & \Bracket{q, q}_{\zz\BPar{i}}
			\end{pmatrix}=
			\begin{pmatrix}
				-2 & 1-i \\
				1+i & -2
			\end{pmatrix}
			\]
			and thus the quadratic form on $D_{4}$ is given by the $\Par{0, 2}$ form
			\[h_{D_{4}}=\Bracket{\Par{z, w}, \Par{z, w}}_{\zz\BPar{i}}=-2\Abs{z}^{2}-2\Abs{w}^{2}+2\Re\Par{z\overline{w}}+2\Im\Par{z\overline{w}}.\]
		\end{itemize}
	\end{proof}
	
	\subsection{Realization of the Deck group via monodromy}
	
	Before proceeding to compute $\text{Im}\Par{\rho_{4}^{-}}$, we will show that the action $T$ on $\cohomoldeg{2}{X}{\zz}$ induced by Deck group is realized via monodromy. By \Cref{Geiser}, there is a loop $\gamma$ realizing the Geiser involution action on $\cohomoldeg{2}{\mathcal{P}}{\zz}$. Then the following holds:
	
	\begin{prop}\label{TrealizedByGamma}
		$\rho_{4}\Par{\gamma}=T$.
	\end{prop}

	To prove this, we show that $\gamma$ realizes the Deck group action on $\mathcal{P}$, and consequently on $X$. We will see this by first observing that $\gamma$ fixes $\pp^{2}$, as it fixes every bitangent to a base point quartic in $\mathcal{U}_{4}$.
	
	\begin{lem}\label{atMost4BitangentsConcur}
		If $n$ bitangents to a smooth quartic curve are concurrent, then $n\leq 4$.
	\end{lem}
	
	\begin{proof}
		Suppose $n$ bitangents are concurrent at a point. We may lift these lines to $n$ concurrent lines in $\mathcal{P}$. Letting $L_{i}$ be one of these $n$ lines, blowing down $L_{i}^{\ast}$ produces a smooth cubic surface $S$ and the remaining $n-1$ lines are concurrent within $S$. Nevertheless, every line in a cubic surface intersects $10$ other lines, which can be separated into $2$ disjoint sets of $5$ pairwise non-intersecting lines. This implies that any $4$ lines in $S$ cannot be pairwise intersecting, and thus $n - 1\leq 3$ or $n \leq 4$, as we wanted.
	\end{proof}
	
	\begin{lem}\label{BitangentsFixed}
		$\gamma$ fixes every point in every bitangent to a given smooth quartic curve.
	\end{lem}
	
	\begin{proof}
		Let $f\in \mathcal{U}_{4}$ be a base point curve, and let $\ell_{1}, \ell_{2}, \ldots, \ell_{28}$ be the $28$ bitangent lines to $V\Par{f}$ in $\pp^{2}$. For each pair of indices $1\leq i < j\leq 28$ let
		\[P_{ij}=\ell_{i}\cap \ell_{j}.\]
		Let $L_{i}, L_{i}^{\ast}$ be the pair of lines in $\mathcal{P}$ lying over the bitangent $\ell_{i}$. The lines $\WBPar{L_{i}, L_{i}^{\ast}}$ are interchanged by $\gamma$, thus leaving each bitangent $\ell_{i}$ fixed. Therefore, $\gamma$ fixes each point $P_{ij}$ in $\pp^{2}$. Any two distinct bitangents to $V\Par{f}$ must intersect due to Bezout's theorem. Some of these bitangents may be concurrent, but at any concurrence point, at most $4$ bitangents may meet by \Cref{atMost4BitangentsConcur}. Hence, $\gamma$ fixes at least $9$ points within each bitangent $\ell_{i}$, therefore fixing all of $\ell_{i}$.
	\end{proof}
	
	Let $\varphi_{\gamma}:V\Par{f}\to V\Par{f}$ be the automorphism on $V\Par{f}$ induced by $\gamma$. Let $\Phi_{\gamma}:\pp^{2}\to \pp^{2}$ be an extension of $\varphi_{\gamma}$ such that $\Phi_{\gamma}\vert_{V\Par{f}}=\varphi_{\gamma}$ (see Theorem 4.2 in \cite{hirose}) and $\Phi_{\gamma}$ agrees with the monodromy action of $\gamma$ on $V\Par{f}$. Namely, $\Phi_{\gamma}$ fixes each bitangent $\ell_{i}$ and $\Phi_{\gamma}$ is an isometry of $\pp^{2}$.
	
	\begin{prop}\label{p^2Fixed}
		$\Phi_{\gamma}$ is the identity map on $\pp^{2}$.
	\end{prop}
	
	\begin{proof}
		Note $\varphi_{\gamma}$ must be the identity map since $\gamma$ fixes all bitangents to $V\Par{f}$, which completely determine the quartic curve curve $V\Par{f}$. At each point $P_{ij}$ in $\pp^{2}$, the differential
		\[D_{P_{ij}}\Phi_{\gamma}:T_{P_{ij}}\pp^{2}\to T_{P_{ij}}\pp^{2}\]
		is the identity map, as it is the identity in the directions spanned by $\ell_{i}$ and $\ell_{j}$ by \Cref{BitangentsFixed}. Since $\Phi_{\gamma}$ is an isometry and $\pp^{2}$ is a complete compact Riemannian manifold, $\Phi_{\gamma}$ must be the identity as well.
	\end{proof}
	
	Hence, $\Phi_{\gamma}$ fixes $\pp^{2}$ and $\gamma$ acts as the Deck transformation $w\mapsto -w$ on $\mathcal{P}=\WBPar{w^{2}=f}$.
	
	\begin{prop}\label{TinMonodromy}
		$\gamma$ acts as the Deck transformation $w\to iw$ on $X=\WBPar{w^{4}=f}$.
	\end{prop}
	
	\begin{proof}
		Choose a point $P$ in $\pp^{2}$ not lying in $V\Par{f}$. The preimages of $P$ via the branched cover $\mathcal{P}\to \pp^{2}$ are
		\begin{align*}
			P_{+} &=\Par{w, P} \\
			P_{-} &=\Par{-w, P}
		\end{align*}
		and the preimages of $P_{+}, P_{-}$ via the branched cover and $X\to \mathcal{P}$ are, respectively
		\begin{align*}
			Q_{1}&=\Par{w, P} & Q_{i}&=\Par{iw, P} \\
			Q_{-1} &=\Par{-w, P} & Q_{-i} &=\Par{-iw, P}
		\end{align*}
		The action of $\gamma$ maps $P_{+}\to P_{-}$ in $\mathcal{P}$ and thus, without loss of generality, $\gamma$ maps
		\begin{center}
			\begin{tikzcd}
				Q_{1}\arrow[r, "\gamma"] & Q_{i}\\
				Q_{-i}& Q_{-1}\arrow[l, "\gamma"]
			\end{tikzcd}
		\end{center}
		Let $t$ be the Deck group action of the branched cover $X\to \pp^{2}$. Then $t^{2}$ is the Deck group action of the branched cover $X\to \mathcal{P}$. In particular, $t^{2}$ fixes the points $P_{+}, P_{-}$ and therefore $t^{2}$ must interchange the pairs of points within the pairs $Q_{1}, Q_{-1}$ and $Q_{i}, Q_{-i}$. This implies, without loss of generality, that $t$ maps the points $Q_{1}, Q_{i}, Q_{-1}, Q_{-i}$ as folows:
		\begin{center}
			\begin{tikzcd}
				Q_{1}\arrow[r, "t"] & Q_{i}\arrow[d, "t"] \\
				Q_{-i}\arrow[u, "t"] & Q_{-1}\arrow[l, "t"]
			\end{tikzcd}
		\end{center}
		Now suppose $\gamma$ maps $Q_{i}\to Q_{1}$ (and therefore $Q_{-i}\to Q_{-1}$). Observing how $t$ and $\gamma$ permute the set of points $\WBPar{Q_{1}, Q_{i}, Q_{-1}, Q_{-i}}$, we see that $t$ acts as the $4-$cycle $(1234)$ while $\gamma$ acts as the involution $(12)(34)$. This is a contradiction because the permutations $(1234)$ and $(12)(34)$ do not commute, while $t$ and $\gamma$ do. Hence, $\gamma$ cannot map $Q_{i}\to Q_{1}$ and thus it maps the points as
		\begin{center}
			\begin{tikzcd}
				Q_{1}\arrow[r, "\gamma"] & Q_{i}\arrow[d, "\gamma"] \\
				Q_{-i}\arrow[u, "\gamma"] & Q_{-1}\arrow[l, "\gamma"]
			\end{tikzcd}
		\end{center}
		coinciding with $t$. Our initial choice of $P$ in $\pp^{2}\setminus V\Par{f}$ was arbitrary, so $\gamma$ coincides with $t$ on all of $\pp^{2}$, concluding that the monodromy action of $\gamma$ is given by that of the Deck transformation $t$.
	\end{proof}
	
	\subsection{The moduli of smooth quartic curves}
	
	The image of $\rho_{4}^{-}$ must lie in the centralizer $C_{L_{-}}\Par{T}$ of $T$ in $L_{-}$ since $\rho_{4}$ and $T$ commute. This is equivalent to respecting the $\zz\BPar{i}$-module structure on $L_{-}$. Moreover, $\rho_{4}^{-}$ is norm-preserving, which translates to the corresponding $\zz\BPar{i}$-lattice automorphisms to be unitary. Therefore
	\[\text{Im}\Par{\rho_{4}^{-}}\subset C_{L_{-}}\Par{T}\cap O\Par{L_{-}}\cong U\Par{h_{L_{-}}}.\]
	For simplicity and consistency with the notation in \cite{kondo}, we let
	\[\Gamma=U\Par{h_{L_{-}}}.\]
	Our goal is to show that $\text{Im}\Par{\rho_{4}^{-}}\cong \Gamma$, thus computing the monodromy group of $\rho_{4}$ by \Cref{reduction}. In order to do this, we appeal to the moduli space of smooth quartic curves $\mathcal{M}_{3}$.
	Theorem 2.5 in \cite{kondo} characterizes this moduli space as the quotient
	\[\mathcal{M}_{3}\cong \Par{\mathcal{D}_{6}-\mathcal{H}}/\Gamma\]
	where $\mathcal{D}_{6}$ is a complex $6$-dimensional ball and $\mathcal{H}$ is a locally finite union of hyperplanes in $\mathcal{D}_{6}$. Following a similar treatment to that in \cite{ACT}, we consider a cover $\widetilde{\mathcal{M}_{3}}$ of $\mathcal{M}_{3}$ as follows:
	
	For a given smooth quartic curve $V\Par{f}$ along with the coresponding $K3$ surface $X_{f}$ lying over it in the fibration $\mathcal{E}_{4, 4}$, define a \emph{framing} $\lambda$ on $V\Par{f}$ as an isomorphism
	\[\lambda: \Par{\zz\BPar{i}^{7}, h_{L_{-}}}\to L_{-}\subset \cohomoldeg{2}{X_{f}}{\zz}.\]
	which respects the intersection form on each lattice. Define $\widetilde{\mathcal{M}_{3}}$ as the moduli space of framed smooth quartic curves, where $\mathbb{P}\Par{\Gamma}=\Gamma/\zz\BPar{i}^{\ast}$ acts on $\widetilde{\mathcal{M}_{3}}$ by precomposition. That is, $\Gamma$ acts on $\widetilde{\mathcal{M}_{3}}$ up to scalar equivalence and for $g\in \Gamma$ we have
	\[g\cdot\Par{V\Par{f}, \lambda}=\Par{V\Par{f}, \lambda\circ g^{-1}}\]
	where framings $\lambda, \lambda'$ differing by a unit in $\zz\BPar{i}$ are equivalent. The composition $j_{\cc}\circ \lambda_{\cc}:\cc^{1, 6}\to V_{i}$ gives a negative-definite hyperplane \[\Par{j_{\cc}\circ \lambda_{\cc}}^{-1}\Par{V_{i}^{1, 1}}\subset \cc^{1, 6}.\]
	These hyperplanes are parametrized by the ball $\mathcal{D}_{6}$ and provide the period map $\wp:\widetilde{\mathcal{M}_{3}}\to \mathcal{D}_{6}$. In \cite{kondo} it is shown that those points in $\mathcal{D}_{6}$ which are not in the image of $\wp$ correspond precisely to roots in $L^{-}$, namely, those $\delta\in L_{-}$ satisfying $\Bracket{\delta, \delta}=-2$. Each root $\delta\in L_{-}$ has an associated hyperplane
	\[H_{\delta}=\WBPar{z\in \mathcal{D}_{6}\mid \Bracket{z, \delta}=0}\subset \mathcal{D}_{6}\] and we denote the union of these hyperplanes by
	\[\mathcal{H}=\bigcup_{\Bracket{\delta, \delta}=-2}H_{\delta}.\]
	By the Torelli theorem for $K3$ surfaces (\cite{shapiroShafarevich}) it follows that the period map $\wp:\widetilde{\mathcal{M}_{3}}\to \mathcal{D}_{6}-\mathcal{H}$ is an isomorphism. This provides a commutative diagram
	
	\begin{center}
		\begin{tikzcd}
			\widetilde{\mathcal{M}_{3}}
			\arrow[r, "\wp", "\cong" ']
			\arrow[d, "\mathbb{P}\Par{\Gamma}" ']& \mathcal{D}_{6}-\mathcal{H}
			\arrow[d, "\mathbb{P}\Par{\Gamma}"] \\
			\mathcal{M}_{3}
			\arrow[r, "\tilde{\wp}" ', "\cong"] & \Par{\mathcal{D}_{6}-\mathcal{H}}/\Gamma
		\end{tikzcd}
	\end{center}

\subsection{Proof of the main theorem}

	A loop in $\mathcal{U}_{4}$ may be regarded as one in $\mathcal{M}_{3}$. This  induces the map $\mu:\pi_{1}\Par{\mathcal{M}_{3}}\to \Gamma $ given by $\mu\Par{\ell}=\rho_{4}^{-}\Par{\ell}$.
	
	\begin{prop}\label{surjectionpi1UZi}
		The monodromy map $\mu: \pi_{1}\Par{\mathcal{M}_{3}}\to \Gamma$ is surjective.
	\end{prop}
	
	\begin{proof}
		Surjectivity to $\mathbb{P}\Par{\Gamma}$ follows from the connectedness of $\widetilde{\mathcal{M}_{3}}$. Fix a smooth quartic curve $V\Par{f}$ and a point $\Par{V\Par{f}, \lambda_{0}}\in \widetilde{\mathcal{M}_{3}}$. For any $g\in \mathbb{P}\Par{\Gamma}$ there is a path $\tilde{\ell}$ starting at $\Par{V\Par{f}, \lambda_{0}}$ and ending at $\Par{V\Par{f}, \lambda_{\ell}}$ such that $\lambda_{\ell} = \lambda_{0}\circ g^{-1}$. Hence, $\tilde{\ell}$ descends to a loop $\ell$ in $\mathcal{M}_{3}$ based at $V\Par{f}$ such that
		\[\mu\Par{\ell}=\lambda_{\ell}^{-1}\circ \lambda_{0}=g.\]
		
		This shows surjectivity to $\Gamma$ up to scalar equivalence, so it remains to show scalars are realized by $\mu$. Indeed, scalars in $\Gamma$ correspond to powers of the action $T$ induced by the Deck transformation $t$ of the branched cover $X\to \pp^{2}$. By \Cref{TrealizedByGamma}, these are realized by a path $\gamma$ inducing the Geiser involution on $\mathcal{P}$, and thus surjectivity to $\Gamma$ follows.
	\end{proof}
	
	It remains to relate $\mu$ to $\pi_{1}\Par{\mathcal{U}_{4}}$. Note that $\mathcal{M}_{3}$ is given by an $PGL_{3}\Par{\cc}$ quotient $q:\mathcal{U}_{4}\to \mathcal{M}_{3}$ as
	\[\mathcal{M}_{3}\cong \mathcal{U}_{4}/PGL_{3}\Par{\cc}.\]
	Moreover, $\mathcal{M}_{3}$ embeds in the moduli space of genus $3$ curves $\mathscr{M}_{3}$ since every smooth quartic curve is a genus $3$ curve. The moduli $\mathcal{M}_{3}$ is obtained by removing the moduli $\mathscr{H}_{3}$ of hyperelliptic genus $3$ surfaces from $\mathscr{M}_{3}$ since every genus $3$ curve is either planar quartic or hyperelliptic. That is,
	\[\mathscr{M}_{3}=\mathcal{M}_{3}\cup \mathscr{H}_{3}.\]
	The moduli $\mathscr{H}_{3}$ is also refered to as the genus $3$ hyperelliptic locus.
	
	\begin{prop}\label{surjectionpi1U4}
		The natural map $q_{\ast}:\pi_{1}\Par{\mathcal{U}_{4}}\to \pi_{1}\Par{\mathcal{M}_{3}}$ induced by the quotient $q$ is surjective.
	\end{prop}
	
	\begin{proof}
		Consider the subvariety $\mathcal{O}\subset \mathscr{M}_{3}$ of genus $3$ curves with a nontrivial automorphism group. Then $\mathcal{O}$ is the union of irreducible components for each topological type of faithful finite group action on a genus $3$ surface $\Sigma_{3}$. Using the Riemann-Hurwitz formula, the dimension over $\cc$ of each component is at most $5$, with equality only holding for the hyperelliptic locus $\mathscr{H}_{3}$. In particular, the remaining components distinct of $\mathscr{H}_{3}$ lie within $\mathcal{M}_{3}$ and are of dimension at most $4$. Since $\mathcal{M}_{3}$ is $6$-dimensional, we have that
		\[\text{codim}_{\mathcal{M}_{3}}\Par{\mathcal{O}-\mathscr{H}_{3}}\geq 2\quad\text{ and }\quad \text{codim}_{\mathcal{U}_{{4}}}\Par{q^{-1}\Par{\mathcal{O}-\mathscr{H}_{3}}}\geq 2.\]
		By excising $q^{-1}\Par{\mathcal{O}-\mathscr{H}_{3}}$ from $\mathcal{U}_{4}$ we obtain a $PGL_{3}\Par{\cc}$ fibration
		\[q:\mathcal{U}_{4}-q^{-1}\Par{\mathcal{O}-\mathscr{H}_{3}}\to \mathcal{M}_{3}-\Par{\mathcal{O}-\mathscr{H}_{3}}.\]
		The long homotopy exact sequence induced by this fibration is given by
		\[\cdots\to \pi_{1}\Par{PGL_{3}\Par{\cc}}\to\pi_{1}\Par{\mathcal{U}_{4}-q^{-1}\Par{\mathcal{O}-\mathscr{H}_{3}}}\to \pi_{1}\Par{\mathcal{M}_{3}-\Par{\mathcal{O}-\mathscr{H}_{3}}}\to \pi_{0}\Par{PGL_{3}\Par{\cc}}=1\]
		implying that the map $\pi_{1}\Par{\mathcal{U}_{4}-q^{-1}\Par{\mathcal{O}-\mathscr{H}_{3}}}\to \pi_{1}\Par{\mathcal{M}_{3}-\Par{\mathcal{O}-\mathscr{H}_{3}}}$ is surjective. Furthermore, we have
		\[
		\pi_{1}\Par{\mathcal{M}_{3}}\cong 	\pi_{1}\Par{\mathcal{M}_{3}-\Par{\mathcal{O}-\mathscr{H}_{3}}}
		\quad\text{ and }\quad
		\pi_{1}\Par{\mathcal{U}_{4}}\cong \pi_{1}\Par{\mathcal{U}_{4}-q^{-1}\Par{\mathcal{O}-\mathscr{H}_{3}}}
		\]
		since the excised spaces are of codimension at least $2$. Therefore, $q_{\ast}: \pi_{1}\Par{\mathcal{U}_{4}}\to \pi_{1}\Par{\mathcal{M}_{3}}$ is surjective.
	\end{proof}
	
	The composition of the maps $\mu\circ q_{\ast}$ gives precisely the monodromy map $\rho_{4}^{-}$. \Cref{surjectionpi1UZi} and \Cref{surjectionpi1U4} combined give surjectivity of $\rho_{4}^{-}$, concluding our computation.
	\printbibliography
	
	\vspace{.5cm}
	\noindent Department of Mathematics, University of Chicago\\
	E-mail: \href{mailto:amedrano@math.uchicago.edu}{amedrano@math.uchicago.edu}
	
\end{document}